\documentclass[11pt]{amsart}

\usepackage{amsmath,amsthm,amssymb}
\usepackage{mathrsfs}
\usepackage{cite}

\makeatletter
\newcommand{\subjclassname@later}{\textup{2010} Mathematics Subject Classification}
\makeatother

\newtheorem{theorem}{Theorem}[section]
\newtheorem{lem}{Lemma}[section]

\newtheorem{corollary}{Corollary}[section]

\newtheorem{remark}{Remark}[section]

\numberwithin{equation}{section} \numberwithin{theorem}{section}

\DeclareMathOperator{\Ree}{Re}

\allowdisplaybreaks

\begin{document}

\title{Sufficient conditions for Janowski starlike functions with fixed second coefficient}

\author[N. M. Alarifi]{Najla M. Alarifi}
\address{Department of Mathematics, Imam Abdulrahman Bin Faisal University, Dammam 31113, Kingdom of Saudi Arabia}
\email{najarifi@hotmail.com}

\begin{abstract}
Let $A, B, D, E \in[-1,1]$ and let $p(z)$ be an analytic function with fixed initial coefficient defined in the open unit disk. Conditions on $A, B, D,$ and $E$ are determined so that $1+ \alpha z p'(z)$ being subordinated to $(1+Dz)/(1+Ez)$ implies that $p(z)$ is subordinated to $(1+Az)/(1+Bz)$ and other similar implications involving $1+ \alpha z p'(z)/p(z)$, $\alpha p^2(z) + \lambda z p'(z) $, $ \alpha p(z) + (1-\alpha) p^2(z) +  \lambda z p'(z)$, and $ (1-\alpha) p(z) + \alpha \left( 1+  z p'(z)/p(z) \right)$. Also, sufficient conditions for Janowski starlikeness with fixed second coefficient are obtained.
\end{abstract}



\maketitle

\section{Introduction and Preliminaries}

Let $\mathcal{A}$ denote the class of analytic functions $f(z) = z + \sum_{n=2}^{\infty} a_{n}z^{n}$ defined in the open unit disk $\mathbb{D}:=\{z \in \mathbb{C}: |z| < 1\}$. For univalent function $f \in \mathcal{A}$, the famous Bieberbach theorem in \cite{Duren} states that $|a_2| \leq  2$ and this bound  yields  the growth and distortion estimates for functions in the class of univalent functions.  It is a fact that the second coefficient plays an important role in univalent function theory. For this reason, there is a continued interest in the investigations on the role of the second coefficient.
 Various subclasses of univalent functions with fixed second coefficients were investigated beginning with Gronwall \cite{Gronwall}.
 Ali \emph{et al}. \cite{Ali2012, Ali2011} have extended the well-known theory of differential subordination \cite{Miller2000} to the functions with preassigned second coefficients. Nagpal and Ravichandran \cite{Nagpal2012} applied the results in \cite{Ali2011} to
 obtain several results for
subclasses of univalent functions by restricting the functions
to have fixed second coefficient. Several applications of theory of differential subordination of functions with initial coefficient were given as \cite{Ravichandran, Mendiratta, Sharma, Ali2017}.

An analytic function $f$ is \emph{subordinate} to an analytic function $g$, written $f(z) \prec g(z)$, if there exists an analytic self-map $w$ of $\mathbb{D}$ with $w(0) = 0$ satisfying $f(z) = g(w(z))$. Let
$\mathcal{S}^*[A,B]$ denote the class of function $f \in \mathcal{A}$ satisfying the subordination $  z f'(z)/f(z) \prec (1+Az)/(1+Bz)$, $\left(-1 \leq  B < A \leq  1\right)$. Functions in $\mathcal{S}^*[A,B]$ are called the Janowski starlike functions \cite{Jano,Pol}. Certain well-known subclasses of starlike functions are special cases of the class $\mathcal{S}^*[A,B]$ for suitable choices of the parameters $A$ and $B$. For $0 \leq  \alpha < 1$, $\mathcal{S}^*[1-2\alpha,-1]=:\mathcal{S}^*_\alpha$ is the familiar class of starlike functions of order $\alpha$, $\mathcal{S}^*[1-\alpha,0]=:\left\{f \in \mathcal{A}:| z f'(z)/f(z)-1|<1-\alpha\right\}=:\mathcal{S}^*(\alpha)$ and $\mathcal{S}^*[\alpha,-\alpha]=\left\{f \in \mathcal{A}:| z f'(z)/f(z)-1|<\alpha|zf'(z)/f(z)+1|\right\}=:\mathcal{S}^*[\alpha]$.

Silverman \cite{Silverman}, Obradovi\'c and Tuneski \cite{Obra2} have studied properties of functions defined in terms of the quotient $(1+zf''(z)/f'(z)-1)/
(zf'(z)/f(z)).$ Works in this direction include those of  \cite{Nuno1,Nuno3,Ravi1,Ravi2}. In fact, Silverman \cite{Silverman} has obtained the order of starlikeness for the functions in the class $F_c$ defined by
\[F_c :=\left\{ f \in \mathcal{A} \left| \frac{1+zf''(z)/f'(z)} {zf'(z)/f(z)}-1 \right| < c ,\quad0 < c \leq  1\right\}.\]

Further, Silverman's result was improved by Obradovi\'c and Tuneski \cite{Obra2} and showed $F_c \subset \mathcal{S}^*[0,-c] \subset \mathcal{S}^*(2/(1+\sqrt{1+8c})$. Later Tuneski \cite{Tuneski} obtained sufficient conditions for the inclusion $F_c \subset \mathcal{S}^*[A,B]$.
 obtained sufficient conditions  for a function $f  \in F_c$ to be in the class $\mathcal{S}^*[A,B]$ and
its subclasses.
That is, if $p(z)=:zf'(z)/f(z)$, then
\begin{equation*}
 1+\frac{z p'(z)} {p^2(z)} \prec 1+cz\quad \text{implies}\quad  p(z) \prec \frac{1+Az} {1+Bz}.
\end{equation*}
Frasin and Darus \cite{Frasin} have shown that
\begin{equation*}
\frac{(zf(z))''} {f'(z)}-\frac{2zf'(z)} {f(z)}\prec \frac{(1-\alpha)z} {2-\alpha}\quad \text{implies}\quad \left| \frac{z^2f'(z)} {f^2(z)}-1 \right|<1-\alpha
\end{equation*}
which is a special case of subordination
\[1+\beta\frac{z p'(z)} {p^2(z)} \prec \frac{1+Dz} {1+Ez} \quad \text{implies}\quad  p(z) \prec \frac{1+Az} {1+Bz}
\] on  $p(z)=:z^2f'(z)/f^2(z)$.
Another special case of the above implications was considered by Ponnusamy and Rajasekaran \cite{Ponnusamy}.
Nunokawa {\em et al.}\cite{Nuno2} have shown that if $p(z)$ is analytic, $p(0)=1$ and $  zp'(z)\prec  z $, then $ p(z)\prec 1+z $. Using this, they have obtained a criterion for a normalized analytic function to be univalent.
 Recently Ali {\em et al.}\cite{Ali}, obtained the condition on the constants $A, B, D, E \in[-1,1]$ and  $ \beta$ so that $ p(z) \prec (1+Az)/(1+Bz)$
when $1+ \beta z p'(z)/p^n(z) \prec (1+Dz)/(1+Ez)$ for $n = 0,1$.

 In this paper, we determine conditions on $A, B, D, E \in[-1,1]$ so that $1+ \alpha z p'(z)\prec (1+Dz)/(1+Ez) \Rightarrow p(z)\prec (1+Az)/(1+Bz)$, where $p(z)$ is an analytic function with fixed second coefficient. Similar results are obtained by considering the expressions $1+ \alpha z p'(z)/p(z)$, $\alpha p^2(z) + \lambda z p'(z) $, $ \alpha p(z) + (1-\alpha) p^2(z) +  \lambda z p'(z)$, and $ (1-\alpha) p(z) + \alpha \left( 1+  z p'(z)/p(z) \right)$. In addition, sufficient conditions for analytic functions to be Janowski starlike are obtained. The results in this paper are derived through several meticulous lengthy computations, and thus in several instances, these computations were validated by use of Mathematica.

Let $\mathcal{A}_{n,b}$ be the subclass of $\mathcal{A}$ consisting of functions of the form
$ f(z) = z + b z^{n+1} + a_{n+2} z^{n+2} + \cdots, $ where $n \in \mathbb{N}$
and $b$ is a fixed nonnegative real number. For a fixed constant $\mu\geq 0,$ and $n \in \mathbb{N},$ let $\mathcal{H}_{\mu,n}$ be the class consists of analytic functions $p$ on $\mathbb{D}$ of the form
\[p(z) = 1 + \mu z^{n} + p_{n+1}z^{n+1}+\cdots.\]
Let $\Omega$ be a subset of $\mathbb{C}$ and the class $\Psi_{\mu,n}[\Omega]$ consists of those functions $\psi : \mathbb{C}^2\rightarrow\mathbb{C}$ that are continuous in domain $D \subset \mathbb{C}^2$ with $(0,1)\in D$, $\psi(0,1)\in \Omega$, and satisfy the admissibility condition $\psi (i \rho , \sigma) \notin \Omega$ whenever $(i \rho , \sigma) \in D$, $\rho \in \mathbb{R}$, and
\begin{equation*} 
\sigma \leq  -\frac{1} {2 } \left(n + \frac{2-\mu} {2 + \mu}\right)(1 + \rho^2).
\end{equation*}
When $\Omega=\left\{ w:\operatorname{Re} w > 0 \right\} $, let $\Psi_{\mu,n}=: \Psi_{\mu,n}[\Omega]$. The following lemma is needed to prove our results.
\begin{lem} \label{lem101} \cite{Ali2011}
Let $p \in \mathcal{H}_{\mu,n}$ with $0 < \mu \leq  2$. Let $\psi\in \Psi_{\mu,n}$ with associated domain $D$. If $(p(z), z p'(z))\in D $ and  $\operatorname{Re}\psi(p(z), z p'(z))> 0$, then $\operatorname{Re}p(z)> 0$ for $z \in \mathbb{D}$.
\end{lem}

\section{Differential subordination}

\begin{lem}\label{lem2.1}
Suppose that $-1 \leq  B < A \leq  1$, $-1 \leq  E < D \leq  1$, and $ \alpha E<0 $. Assume that
\begin{align}\label{eq0}
 & G = n\left( 2+\mu' \right) + \left(2-\mu' \right),\nonumber \\
  &  H = (n-1) \left( 2+\mu' \right)+\left( 2-\mu' \right),
\end{align}
where $0<\mu'= 2\mu /(A-B) \leq  2 .$
In addition for all $n>1$, let
\begin{align}\label{eq00}
   \left( (D-E) (1-B)^2 + \alpha E (A-B) \right)^2 G^2  + (D-E)^2 (1+B)^4 H^2 \nonumber\\
 \leq    \alpha^2 (A-B)^2 G^2 + 2 (D-E)(1+B)^2 & \big((D-E) (1-B)^2 - \alpha E (A-B) \big)G H.
\end{align}
If $p\in \mathcal{H}_{\mu,n}$ and
\begin{equation}
1+ \alpha z p'(z) \prec \frac{1+Dz}{1+Ez},\label{1}
\end{equation}
then
\begin{align*}
p(z) \prec \frac{1+Az}{1+Bz}.
\end{align*}
\end{lem}

\begin{proof}
Define the function $q : \mathbb{D} \rightarrow \mathbb{C}$ by
\begin{equation}
q(z) = \frac{(A-1)+(1-B)p(z)}{(A+1)-(1+B)p(z)}, \label{eq2}
\end{equation}
where $p\in \mathcal{H}_{\mu,n}$.
Then $q$ is analytic on $\mathbb{D}$ and $q(z) = 1 + 2 \mu' z^{n} + q_{n+1} z^{n+1}+\cdots \in \mathcal{H}_{\mu',n},$ where $0<\mu'= 2\mu /(A-B) \leq  2  .$ It follows from \eqref{1}  and \eqref{eq2} that
\begin{align*}
 \Ree \left\{
\frac {(D-E) \big((1+B) q(z) + (1-B)\big)^2  + 2 \alpha (1-E) (A-B) z q'(z)}
{(D-E) \big((1+B) q(z) + (1-B)\big)^2  - 2 \alpha (1+E) (A-B) z q'(z)}
\right\}
> 0.
\end{align*}
Define $\psi:\mathbb{C}^2\rightarrow \mathbb{C}$ by
\begin{equation*}
\psi(r,s)=
\frac
{(D-E) \big((1+B) r + (1-B)\big)^2  + 2 \alpha (1-E) (A-B) s}
{(D-E) \big((1+B) r + (1-B)\big)^2  - 2 \alpha (1+E) (A-B) s}.
\end{equation*}
Then $\psi $ is continuous of r and s on $D =: \mathbb{C}^2-\{(r,s): N(r,s)=0\},$ where $N(r,s)=(D-E) \big((1+B) r + (1-B)\big)^2  - 2 \alpha (1+E) (A-B) s $. Note that
 $(1, 0) \in D $ and $\Ree\left\{\psi (1, 0)\right\}> 0$. It also follows that for all
 $(i \rho , \sigma) \in D,$
\begin{align*}
&\Ree \left\{\psi(i \rho, \sigma)\right\}=\\
&\quad \quad\quad \quad \Ree \left\{\frac
{(D-E)(1-B)^2 + 2(1-B^2)(D-E) i \rho - (D-E)(1+B)^2 \rho^2 + 2 \alpha (1-E) (A-B) \sigma}
{(D-E)(1-B)^2 + 2(1-B^2)(D-E) i \rho - (D-E)(1+B)^2 \rho^2 - 2 \alpha (1+E) (A-B) \sigma}
\right\}.
\end{align*}
For ease in computations, let
\begin{align*}
& a = (D-E)(1-B)^2 ,\\
& b = 2 \alpha (1-E) (A-B) ,\\
& c = - (D-E)(1+B)^2 ,\\
& d = 2(D-E)(1-B^2) ,\\
& e = - 2 \alpha (1+E) (A-B).
\end{align*}
Then $ \Ree \psi(i \rho, \sigma)= \Ree \left\{(a + b\sigma + c \rho^2 + d i \rho)/(a + e\sigma + c \rho^2 + d i \rho )\right\}$.
For $\Ree \psi(i \rho, \sigma)<0$, we need to show
\[a^2 + a(b+e)\sigma +  be \sigma^2 + \left( 2ac+d^2 +  c( b+e)\sigma \right)  \rho^2 + c^2 \rho^4< 0.\]
Since $\sigma \leq  -1/2,$ and  $\alpha E < 0,$ then
\begin{align}\label{eq2.4}
2ac + d^2 + (b+e)c \sigma & \geqslant 2ac + d^2 - \frac{1}{2} (b+e)c \nonumber \\
&=2  (D-E)(1+B)^2  \left((D-E) (1-B)^2 - \alpha E (A-B) \right)>0.
\end{align}
Also,
\[\rho^2   \leq  - \big( \frac{2(2+\mu') \sigma} {n(2+\mu') + (2 - \mu')} + 1 \big),\] where $\mu'=2 \mu/(A-B)$, and hence if \eqref{eq2.4} holds, then
\begin{align}\label{eq2.5}
 & a^2 + a(b+e) \sigma + b e \sigma^2 + \left( 2 a c + d^2 + (b+e) c \sigma \right) \rho^2 + c^2 \rho^4 \nonumber \\
& \leq    \Big(be-\frac{2(2+\mu')(b+e)c}{(2+\mu') n + (2-\mu')} + \frac{4 (2+\mu')^2 c^2 }{\big((2+\mu') n + (2-\mu')\big)^2} \Big) \sigma^2 \nonumber \\
    &\quad + \Big( (a-c)(b+e)+(2ac+d^2)
                \left( \frac{-2(2+\mu')}{(2+\mu') n + (2-\mu')} \right) \Big.\nonumber\\&\quad
                 \Big.+  \frac{4(2+\mu')c^2}{(2+\mu') n + (2-\mu')}\Big)  \sigma + \big( a^2 - 2ac - d^2 + c^2 \big).
 \end{align}
In view of \eqref{eq2.5}, let $f(\sigma) =: x \sigma^2 + y \sigma + z$,
where
\begin{align*}
 &x =  be-\frac{2(2+\mu')(b+e)c}{(2+\mu') n + (2-\mu')} + \frac{4 (2+\mu')^2 c^2 }{\big((2+\mu') n + (2-\mu')\big)^2}, \\
 &y = (a-c)(b+e)+(2ac+d^2) \left( \frac{-2(2+\mu')}{(2+\mu') n + (2-\mu')} \right) +  \frac{4(2+\mu')c^2}{(2+\mu') n + (2-\mu')}, \\
 &z =  a^2 + c ( c-2a ) - d^2.
\end{align*}

Since
\begin{align}\label{eq1.1}
&f(\sigma) \leq  \max_{\sigma \leq  - \frac{1}{2}} f(\sigma)=\max \left\{ f \left( - \frac{1}{2} \right) \ , \  f \left( - \frac{y}{2x} \right) \right\}=\frac{4z-2y+x}{4},
\end{align}
then $ \max_{\sigma \leq  - \frac{1}{2}} f(\sigma)\leq0$  is equivalent to
\begin{align*}\label{eq3}
  4z-2y+x&=\big(4a^2+be-2a(b+e)\big)\big(n (2+\mu') + (2-\mu')\big)^2 \nonumber \\
 &\quad -2\big(2d^2+4ac-(b+e)c\big) \big(n(2+\mu')+(2-\mu')\big)\big((n-1) (2+\mu') + (2-\mu')\big) \nonumber \\
 & \quad + 4c^2 \big((n-1)(2+\mu') + (2-\mu')\big)^2 \leq 0.
\end{align*}
On substituting  the values of $a$,$b$,$c$,$d$ and $e$ in above expression, we obtain
\begin{align*}
f(\sigma)&\leq   4\left( \big((D-E) (1-B)^2 + \alpha E (A-B) \big)^2- \alpha^2 (A-B)^2  \right)G^2+4 (D-E)^2 (1+B)^4 H^2 \\
&\quad - 8 (D-E)(1+B)^2  \big((D-E) (1-B)^2 - \alpha E (A-B) \big) G H \\
&\leq  0.
\end{align*}

If inequality \eqref{eq00} holds, then
this implies that the function $ \psi(r,s)$ satisfies the conditions in Lemma \ref{lem101}, which completes the proof.
\end{proof}
\begin{remark}\label{rem2.1}\
\begin{enumerate}
\item For $n=1$, and $\mu=  A-B$, Lemma \ref{lem2.1} reduces to \cite[Lemma 2.1]{Ali} for $\alpha=\beta$.
\item  When $n=1$, $\mu=|p_1|\leq A-B$, $\alpha=1$, $E=0=B$, and $ D=1=A$, Lemma \ref{lem2.1} reduces to \cite[Lemma 1]{Nuno2}.
\end{enumerate}
\end{remark}
%
Next theorem follows from Lemma \ref{lem2.1} by taking
 $p(z)=zf'(z)/f(z),$ where $f\in \mathcal{A}_{n,b}$ and $0<\mu=nb\leq A-B.$
 \begin{theorem}\label{thm2.1} Let the conditions of Lemma 2.1 hold. If $f\in \mathcal{A}_{n,b}$ satisfies
\[1+\alpha\frac{zf'(z)}{f(z)}\left(1+\frac{zf''(z)}{f'(z)}-\frac{zf'(z)}{f(z)}\right)\prec \frac{1+Dz}{1+Ez},\]
then $f\in S^*[A,B].$
\end{theorem}
For $0\leq  \lambda < 1$ and $\delta \leq  1$, by taking $\alpha=1,$ $A=-B=\lambda,$ and $D=-E=\delta $ in Theorem 2.1, we get the following result.
\begin{corollary}\label{crl1}
Let $0\leq  \lambda <1$. If $f\in \mathcal{A}_{n,b}$ satisfies
\[\left|\frac{zf'(z)}{f(z)}\left(1+\frac{zf''(z)}{f'(z)}-\frac{zf'(z)}{f(z)}\right)\right|<\delta \left|2+\frac{zf'(z)}{f(z)}\left(1+\frac{zf''(z)}{f'(z)}-\frac{zf'(z)}{f(z)}\right)\right|,\]
where $\delta=(\lambda G)/\sqrt{((1+\lambda)^2-\lambda)^2 G^2 +(1-\lambda)^4H^2-2(1-\lambda)^2((1+\lambda)^2+\lambda)GH}$.
Then $f\in S^*[\lambda].$
\end{corollary}
If we take $n=1 $ and $\mu=A=-B,$ then Corollary \ref{crl1} reduces to \cite[Corollary 2.4]{Ali} for $\alpha=\lambda$.
\begin{corollary}\label{crl2}
If $f\in \mathcal{A}_{n,b}$ satisfies
\[\left|\frac{zf'(z)}{f(z)}\left(1+\frac{zf''(z)}{f'(z)}-\frac{zf'(z)}{f(z)}\right)\right|< \frac{1-\lambda}{2},\]
then $f\in S^*_\lambda.$
\begin{proof}
By letting $\alpha=1$, $A=1-2\lambda$, $B=-1$, $E=0$, and $D=(1-\lambda)/2$, where $0 \leq  \lambda < 1$ in Theorem \ref{thm2.1}, we get the required result.
\end{proof}
\end{corollary}
\begin{remark}\label{rem2.3}
For $n=1 $  and $ \mu= A-B,$ then Corollary \ref{crl2} reduces to \cite[Corollary 2.5]{Ali} for $\alpha=\lambda.$
\end{remark}
\begin{corollary}\label{crl3} If $f\in \mathcal{A}_{n,b}$ satisfies
\[\left|\frac{zf'(z)}{f(z)}\left(1+\frac{zf''(z)}{f'(z)}-\frac{zf'(z)}{f(z)}\right)\right|< \delta,\]
where $\delta=(1-\lambda)G/(G-H) $, and $0 \leq  \lambda < 1$. Then $f\in S^*(\lambda)$.
\end{corollary}
\begin{proof}
The result follows from Theorem \ref{thm2.1} by taking $\alpha=1$, $A=1-\lambda$, $B=0=E$, and $D=\delta.$
\end{proof}

When $n=1$ and $b=A-B,$ Corollary \ref{crl3} reduces to the following result.

\begin{corollary}\label{crl4} If $f\in \mathcal{A} $ satisfies
  \[\left|\frac{zf'(z)}{f(z)}\left(1+\frac{zf''(z)}{f'(z)}-\frac{zf'(z)}{f(z)}\right)\right|<1-\lambda.\]
Then $f\in S^*(\lambda)$.
\end{corollary}
\begin{lem}\label{lem2.2}
Suppose that $-1 \leq  B < A \leq  1$, $0 < E < D \leq  1$, and $ \alpha E<0 $. Also let $G$ and $H$ as in \eqref{eq0}.
In addition for all $n>1$, let
\begin{align}\label{eq2-11}
& \left( (D-E) (1+A)^2 - E (A-B) \right)^2 G^2 + (D-E)^2 (1-A)^4 H^2 \nonumber \\
&\leq   (A-B)^2 G^2 + 2 (D-E)(1-A)^2  \big((D-E) (1+A)^2 + E (A-B) \big) G H.
\end{align}
If $p\in \mathcal{H}_{\mu,n}$ and
\[1+ \frac{z p'(z) }{p^2(z)} \prec \frac{1+Dz}{1+Ez},\]
then
\[p(z) \prec \frac{1+Az}{1+Bz}.\]
\end{lem}
\begin{proof}
The proof is as similar as Lemma \ref{lem2.1} by replacing $p(z)$ by $1/p(z)$, $A$ by $-B,$ $B$ by $-A,$ and $\alpha=-1$.
\end{proof}

Next theorem follows from Lemma \ref{lem2.2} by taking $p(z)=zf'(z)/f(z),$ where $f\in \mathcal{A}_{n,b}$ and $0<\mu=nb\leq A-B.$
\begin{theorem}\label{thm2.2}
Let the conditions of Lemma \ref{lem2.2} hold. If $f\in \mathcal{A}_{n,b}$ satisfies
\[\frac{1+zf''(z)/f'(z)}{zf'(z)/f(z)}\prec \frac{1+Dz}{1+Ez},\]
then $f\in S^*[A,B]$.
\end{theorem}

\begin{remark}\label{rem2.4}\
\begin{enumerate}
\item The special case $n=1,$ $b=A-B$, $E=0$, and $D=c$, where $0<c=(A-B)/(1+|A|)^2$, Theorem 2.2 reduces to \cite[Corollary 2.6]{Tuneski}.
\item  If $n=1,$ $b=A-B,$ $A=0=E$, and $D=c=-B$, where $0<c \leq  1$ Theorem 2.2 reduces to \cite[Theorem 1]{Obra2}.
\item When $n=1,$ $b=A-B,$ $A=0=E$, and $D=1=-B$, Theorem 2.2 reduces to \cite[Corollary 1]{Silverman}.
\end{enumerate}
\end{remark}

\begin{corollary}
If $f\in F_{c}$, where $c=(1-\lambda)G/|(2-\lambda)^2G-\lambda^2H|,$ where $0 \leq  \lambda < 1$, then $f\in S^*(\lambda)$.
\end{corollary}
\begin{proof}
The result is derived from Theorem \ref{thm2.2} by taking  $A=1-\lambda ,$ and $B=0= E.$
\end{proof}
With $n=1 $ and $ \mu= A-B,$ we can conclude that if $f\in F_{(1-\lambda)/(2-\lambda)^2}$, then $f\in S^*(\lambda)$.
\begin{corollary}\label{crl2.5}
 If $f\in \mathcal{A}_{n,b}$ satisfies
\begin{align}\label{eq2-12}
\left|1+\frac{zf''(z)}{f'(z)}-\frac{zf'(z)}{f(z)}\right|<\delta\left|1+\frac{zf''(z)}{f'(z)}+\frac{zf'(z)}{f(z)}\right|,
\end{align}
where $\delta=(\lambda G)/\sqrt{((1+\lambda)^2+\lambda)^2 G^2 +(1-\lambda)^4H^2-2(1-\lambda)^2((1+\lambda)^2-\lambda)GH}$, and $0 \leq  \lambda < 1$,
then $f\in S^*[\lambda]$.
\begin{proof}
The proof follows from Theorem \ref{thm2.2} by taking  $A=\lambda=-B,$ and $D=\delta=-E.$
\end{proof}
With $n=1,$ and $\mu=A-B$ observe that if \eqref{eq2-12} holds with
$\delta=\lambda /(1+3\lambda+\lambda^2)$, then $f\in S^*[\lambda]$.
\end{corollary}
\begin{corollary}\label{crl2.6}
If \eqref{eq2-12} holds with
\[\delta=((1-\lambda) G)/\sqrt{(1-\lambda)^2(5-4\lambda)^2 G^2 +16\lambda^2 H^2-8 \lambda^2(1-\lambda)(3-4\lambda)GH},\]
where $0\leq\lambda < 1$, then $f\in S^*_\lambda$.
\begin{proof}
The result is derived from Theorem  \ref{thm2.2} by taking  $A=1-2\lambda,$ $ B=-1,$ and $D=\delta=-E.$
\end{proof}
With  $n=1,$  and $\mu=A-B$ conclude that if  \eqref{eq2-12} holds with $ \delta=1/(5-4\lambda)$, where $0 \leq  \lambda < 1$, then $f\in S^*_\lambda$.
\end{corollary}

\begin{lem} \label{lem2.3}
Suppose that $-1 \leq  B < A \leq  1$, $-1 \leq  E < D \leq  1$, and $ \alpha E<0 $. Also let $G$ and $H$ as in \eqref{eq0}.
In addition, for all $n>1$, let
\begin{align*}
&\left( \big( (D-E)(1-A) (1-B) + \alpha E (A-B) \big)^2-\alpha^2 (A-B)^2\right) G^2 + (D-E)^2 (1+A)^2 (1+B)^2 H^2 \nonumber \\
  \ & \leq  2 (D-E) \Big( (D-E) \big( (1-AB)^2 + (A-B)^2 \big)  + \alpha E (A-B)(1+A)(1+B) \Big) G H.
 \end{align*}
If $p\in \mathcal{H}_{\mu,n}$ and
\begin{equation}
 1+ \alpha \frac{z p'(z)}{p(z)} \prec \frac{1+Dz}{1+Ez}, \label{eq4}
\end{equation}
then
\begin{align*}
 p(z) \prec \frac{1+Az}{1+Bz}.
\end{align*}
\end{lem}
\begin{proof}
Define the function $q : \mathbb{D} \rightarrow \mathbb{C}$ as \eqref{eq2}.
Then $q$ is analytic on $\mathbb{D}$ and $q(z) = 1 + 2 \mu' z^{n} + q_{n+1} z^{n+1}+\cdots \in \mathcal{H}_{\mu',n},$ where $0<\mu'= 2\mu /(A-B) \leq  2  .$ It follows from
\eqref{eq2}  and \eqref{eq4} that
{\small\begin{align*}
\Ree \left\{
\frac
{(D-E) \big((1-B)+(1+B)q(z)\big)\big((1-A)+(1+A)q(z)\big) +  2\alpha(1-E)(A-B)z q'(z)}
{(D-E) \big((1-B)+(1+B)q(z)\big)\big((1-A)+(1+A)q(z)\big)  - 2\alpha(1+E)(A-B)z q'(z)}
\right\}
> 0.
\end{align*}}
Define $\psi:\mathbb{C}^2\rightarrow \mathbb{C}$ by
\begin{equation*}
\psi(r,s)=
\frac
{(D-E) \big((1-B) + (1+B)r\big)\big((1-A)+(1+A)r\big)  + 2 \alpha (1-E) (A-B) s}
{(D-E) \big((1-B) + (1+B)r\big)\big((1-A)+(1+A)r\big)  - 2 \alpha (1+E) (A-B) s}.
\end{equation*}
Then $\psi $ is continuous of r and s on $D =:\mathbb{C}^2-\{(r,s): N'(r, s)= 0\}$, where $N'(r,s)=:(D-E) \big((1-B) + (1+B)r\big)\big((1-A)+(1+A)r\big)  - 2 \alpha (1+E) (A-B)s$. Note that
 $(1, 0) \in D $ and $\Ree\left\{\psi (1, 0)\right\}> 0$. It also follows that for all
 $(i \rho , \sigma) \in D,$
\begin{align*}
&\Ree \left\{\psi(i \rho, \sigma)\right\}= \Ree \left\{
\frac
{(D-E) [ (1-B)+(1+B) i \rho][(1-A)+(1+A)i \rho]  + 2 \alpha (1-E) (A-B) \sigma}
{(D-E) [ (1-B)+(1+B) i \rho][(1-A)+(1+A)i \rho]  - 2 \alpha (1+E) (A-B) \sigma}
\right\}.
\end{align*}
Let
\begin{align*}
&a = (D-E)(1-A)(1-B), \\
&b = 2 \alpha (1-E) (A-B), \\
&c = - (D-E)(1+A)(1+B), \\
&d = 2 (D-E)(1-AB), \\
&e = - 2 \alpha (1+E) (A-B).
\end{align*}
Then $ \Ree \psi(i \rho, \sigma)= \Ree \left\{(a + b\sigma + c \rho^2 + d i \rho)/(a + e\sigma + c \rho^2 + d i \rho )\right\}$.
For $\Ree \psi(i \rho, \sigma)<0,$ we need to prove
\[a^2 + a(b+e)\sigma +  be \sigma^2 + \left(2ac+d^2 +  c( b+e)\sigma \right)  \rho^2 + c^2 \rho^4< 0.\]
Since $\sigma \le -1/2,$ it follows that
\begin{align}\label{eq2.6}
  2ac + d^2 + (b+e)c \sigma  &\geqslant  2ac + d^2 - \frac{1}{2} (b+e)c \nonumber \\
 & = 2(D-E)\big((D-E) \left( (1-AB)^2 + (A-B)^2 \right)\nonumber \\
 &\quad - \alpha E  (A-B) (1+A)(1+B)\big)\leq 0,
 \end{align}
provided $\alpha E < 0$. Also,  $\rho^2   \leq  - \big( \left( 2 (2+\mu') \sigma \right) / \left(n(2+\mu')\right.$ $\left. + (2 - \mu') \right) + 1  \big)$ where $\mu'=2 \mu/(A-B)$, then the proof follows on lines similar to Lemma \ref{lem2.1}.
\end{proof}

\begin{remark}\label{rem2.5}\
\begin{enumerate}
\item For $n=1$, $\mu=A-B$, and $\alpha=-1$, $A=\lambda= E$, where $|\lambda|\leq  1$, and $D=B= 0$  Lemma \ref{lem2.3} reduces to \cite[Theorem 1(iii)]{Ponnusamy}.
\item  When $n=1$, and $\mu=A-B$, Lemma \ref{lem2.3} reduces to \cite[Lemma 2.10]{Ali} for $\alpha=\beta$.

\end{enumerate}
\end{remark}
%

The following result is derived from Lemma \ref{lem2.3} by letting $\alpha=1$, $B= 0=E$, and $D=(AG)/$\\ $\sqrt{(1-A)^2 G^2+(1+A)^2 H^2 -2(1+A^2)GH}$.
\begin{theorem}\label{thm2.8}
Let $(0<A \leq  1)$. If $p \in \mathcal{H}_{\mu,n},$ and
\[|zp'(z)/p(z)|<(AG)/\sqrt{(1-A)^2 G^2+(1+A)^2 H^2 -2(1+A^2)GH},\] then $p(z)\prec 1+Az$.
\end{theorem}
With $n=1,$ by taking $p(z)=zf'(z)/f(z)$ where $f\in \mathcal{A}_b,$ and $b=A-B,$ then for $A=1-\lambda,$ theorem \ref{thm2.8} reduces to the following result
 \begin{corollary}\label{crl2.8}
If $f\in \mathcal{A} $ satisfies
\[\left|1+\frac{zf''(z)}{f'(z)}-\frac{zf'(z)}{f(z)}\right|<\frac{1-\lambda}{\lambda},\quad (0\leq  \lambda <1),\]
then $f(z)\in S^*(\lambda)$.
\end{corollary}

\begin{lem}\label{lem2.4}
Suppose that $-1 \leq  B < A \leq  1$, $-1 \leq  E < D \leq  1$,
\begin{align*}
& K= \big((D-1)(1+B)^2 - \alpha E(1+A)^2 \big)^2- \alpha^2 (1+A)^4\geq0, \\
& L=\lambda(A-B)  \big( E (D-1)(1+B)^2 + \alpha (1-E^2)(1+A)^2 \big)< 0,\\
& M= \big((D-1) (1-B^2) - \alpha E (1-A^2) \big)^2 + 4 \alpha E (D-1) (A-B)^2 - \alpha^2 (1-A^2)^2\\
&\quad \quad  - \lambda (A-B)\big( \alpha (1-E^2)  (1+A)^2 + E (D-1) (1+B)^2  \big)>0,\\
& N= \big(   (D-1) (1-B)^2 -  \alpha E  (1-A)^2 \big)^2- \big(\alpha (1-A)^2 - \lambda (A-B)\big)^2\\
 &\quad\quad  - \lambda (A-B) \Big(  E^2 \left( 2\alpha (1-A)^2 - \lambda (A-B) \right) - 2 E (D-1) (1-B)^2 \Big).
\end{align*}
Also let $G$ and $H$ as in \eqref{eq0}. In addition, for all $n>1$, let
\begin{align}
 &  N G^2-2 M GH +K H^2 \leq  0 \label{eq6'}
\end{align}
If $p\in \mathcal{H}_{\mu,n}$ and
\begin{equation}
 \alpha p^2(z) + \lambda z p'(z) \prec \frac{1+Dz}{1+Ez}, \label{eq6}
\end{equation}
then
\[p(z) \prec \frac{1+Az}{1+Bz}.\]
\end{lem}
\begin{proof}
It follows from \eqref{eq2}  and \eqref{eq6} that

\[\Ree \left\{
\psi(q(z),zq'(z))\right\}=\Ree \left\{
\frac
{P(q(z),zq'(z))}{Q(q(z),zq'(z))}\right\}
> 0,\]
where
\begin{align*}
P(q(z),zq'(z))&=(D-1) \big((1+B)q (z)+(1-B)\big)^2\\
&\quad+ \alpha\left( (1-E) \big(  (1+A)q^2(z)+(1-A)\big)^2+ 2 \lambda z(A-B) q'(z) \right),
\end{align*}

 and
\begin{align*}
P(q(z),zq'(z))&=(D-1) \big((1+B)q (z)+(1-B)\big)^2\\
&\quad- \alpha\left( (1+E) \big(  (1+A)q^2(z)+(1-A)\big)^2+ 2 \lambda z(A-B) q'(z) \right).
\end{align*}

Define $\psi:\mathbb{C}^2\rightarrow \mathbb{C}$ by
\begin{align*}
\psi(r,s)=:\frac{P(r,s)}{Q(r,s)},
\end{align*}
where
\begin{align*}
P(i \rho, \sigma)&=(D-1)(1-B)^2 + \alpha (1-E) (1-A)^2+ 2 \lambda (1-E) (A-B)\sigma \\
&\quad -\Big((D-1)(1+B)^2 + \alpha (1-E)(1+A)^2\Big) \rho^2 +2 \Big( (D-1)(1-B^2)\\
&\quad +  \alpha (1-E)(1-A^2)\Big) i \rho,
\end{align*}
and
\begin{align*}
Q(i \rho, \sigma)&=(D-1)(1-B)^2 - \alpha (1+E) (1-A)^2- 2 \lambda  (1+E)(A-B)\sigma\\
&\quad -\Big((D-1)(1+B)^2 - \alpha (1+E)(1+A)^2 \Big) \rho^2 \\
&\quad +2 \Big((D-1)(1-B^2)-\alpha (1+E)(1-A^2)\Big) i \rho.
\end{align*}
Then $\psi $ is continuous of r and s on $D =:\mathbb{C}^2-\{(r,s):Q(r,s)=0 \}$. Note that
 $(1, 0) \in D $ and $\Ree\left\{\psi (1, 0)\right\}> 0$. It also follows that for all
 $(i \rho , \sigma) \in D,$ $ \Ree \left\{\psi(i \rho, \sigma)\right\}=\Ree \left\{
P(i \rho, \sigma)/Q(i \rho, \sigma)\right\}.$

Let
\begin{align*}
& a =  (D-1)(1-B)^2 + \alpha (1-E) (1-A)^2, \\
& b =  2 \lambda (1-E) (A-B), \\
& c =  - (D-1)(1+B)^2 - \alpha (1-E)(1+A)^2, \\
& d =  2 (D-1)(1-B^2) + 2 \alpha (1-E)(1-A^2), \\
& e =  (D-1)(1-B)^2 - \alpha (1+E) (1-A)^2, \\
& f =  - 2 \lambda  (1+E) (A-B), \\
& g =  -  (D-1)(1+B)^2 + \alpha (1+E)(1+A)^2 ,\\
& h =    2 (D-1)(1-B^2) - 2 \alpha (1+E)(1-A^2).
\end{align*}
For $\Ree \psi(i \rho, \sigma)<0$ , our claim is
\[ae + (af+be)\sigma + \big((ag+ce+hd) + (bg+cf)\sigma \big)\rho^2 + cg\rho^4 + fb \sigma^2 < 0.\]
Since $\sigma \leq  -1/2,$ then
\begin{align}\label{eq2.7}
   ag+ce+hd+(bg+cf)\sigma &\geq ag+ce+hd-\frac{1}{2}(bg+cf)\nonumber \\
  & =  2 \Big( \big( (D-1) (1-B^2) - \alpha E (1-A^2) \big)^2  + 4 \alpha E (D-1) (A-B)^2 \nonumber \\
   &\quad - \alpha^2 (1-A^2)^2- \lambda (A-B) \big(  E (D-1) (1+B)^2  \big.\nonumber \\
 \big.&\quad +\alpha (1-E^2)  (1+A)^2 \big)\Big)>0.
\end{align}

Observe that  $\rho^2   \leq  - \big( \left( 2 (2+\mu') \sigma \right) \big)/ \big(\left(n(2+\mu')+ (2 - \mu') \right) + 1  \big)$ where $\mu'=2 \mu/(A-B)$.
Hence if \eqref{eq2.7} holds and $ \big((D-1)(1+B)^2 - \alpha E(1+A)^2 \big)^2 - \alpha^2  (1+A)^4 >0$, then
\begin{align}\label{eq2.8}
   &  ae+(af+be)\sigma + fb\sigma^2 + \left((ag+ce+hd)+(bg+cf)\sigma \right) \rho^2 + cg\rho^4  \nonumber \\
 & \le \ \Big(fb- (bg+cf) \left(\frac{2 (2+\mu') }{n(2+\mu') + (2 - \mu') }\right)
       + cg \left(  \frac{4 (2+\mu')^2 }{(n(2+\mu') + (2 - \mu'))^2 } \right)  \Big) \sigma^2\nonumber \\
    &\quad + \Big(af+be-bg-cf+ (2cg-ag-ce-hd) \left(\frac{2 (2+\mu') }{n(2+\mu') + (2 - \mu') }\right) \Big) \sigma \nonumber \\
    & \quad + ae-ag-ce-hd+cg .
\end{align}
Let
\begin{align*}
& x'  =   \Big(  cg \left(  \frac{4 (2+\mu')^2 }{(\left(n(2+\mu') + (2 - \mu')\right)^2 } \right)
         - (bg+cf) \left(\frac{2 (2+\mu') }{n(2+\mu') + (2 - \mu') }\right) + bf \Big), \\
& y'  =  \Big(\left(2cg-ag-ce-hd\right) \left(\frac{2 (2+\mu') }{n(2+\mu') + (2 - \mu') }\right) + af+be-bg-cf\Big), \\
& z'  = ae-ag-ce-hd+cg.
\end{align*}
In view of \eqref{eq2.8}, consider
\[f(\sigma) =: x' \sigma^2 + y' \sigma + z'.\]
By using \eqref{eq1.1},
then $ f(\sigma)\leq  0$  if $ 4z'-2y'+x' \leq  0$   or equivalently
\begin{align*}
&\big(ae - \frac{1}{2}af - \frac{1}{2}be + \frac{1}{4}bf \big)  \big(n(2+\mu') + (2 - \mu')\big)^2 \nonumber \\
           & +\big(\frac{1}{2}cf+\frac{1}{2}bg-hd-ce-ag \big)
                         \big( (n-1)(2+\mu')+(2-\mu')\big) \nonumber \\
           &\big(n(2+\mu')+ (2-\mu')\big)+ cg \big( (n-1) (2+\mu') + (2-\mu')  \big)^2\leq  0,
\end{align*}
if inequality \eqref{eq6'} holds.
\end{proof}

\begin{theorem}\label{thm2.3}
Let the conditions of Lemma \ref{lem2.4} hold. If $f\in \mathcal{A}_{n,b}$ satisfies
\[\alpha \left( \frac{zf'(z)}{f(z)}+\frac{z^2f''(z)}{f(z)}\right) \prec \frac{1+Dz}{1+Ez},\]
then $f\in S^*[A,B]$.
\end{theorem}
\begin{proof}
The theorem follows on substituting $\lambda=\alpha ,$ and $p(z)=zf'(z)/f(z),$ where $f\in \mathcal{A}_{n,b} $ and $0<\mu=nb\leq A-B$ in Lemma \ref{lem2.4}.
\end{proof}

 \begin{corollary}\label{crl2.9}
If $f\in \mathcal{A} $ satisfies
\[{\rm Re\ } \left(\frac{zf'(z)}{f(z)}+\frac{z^2f''(z)}{f(z)}\right)> \delta,\]
then $f\in S^*.$
 \end{corollary}
\begin{proof}
By letting $n=1$, $\alpha=1$, $b=A-B$, $A=1$, $B=-1=E$, and $D=1-2 \delta,$ $0 \leq \delta <1$ in Theorem \ref{thm2.4}, we get the required result.
\end{proof}

\begin{lem}\label{lem2.5}
Suppose that $-1 \leq  B < A \leq  1$, $-1 \leq  E < D \leq  1$,
\begin{align*}
&K=\big( (D-1) (1+B)^2  - (1-\alpha) E (1+A)^2 \big)^2- (1-\alpha)^2 (1+A)^4  - \alpha (1+A) (1+B)\\
&\quad\quad  \times \big( 2 E (1+B)^2 + \alpha (1-E^2) (1+A) (1+B) + 2 (1-\alpha)(1-E^2) (1+A)^2 \big)\geq 0,\\
&L=\lambda(A-B) \big( E (D-1)(1+B)^2 + \alpha (1-E^2)(1+A) (1+B)\\
&\quad\quad + (1-\alpha)(1-E^2)(1+A)^2 \big) < 0\\
&M= \Big( \big( (D-1) (1-B^2 ) - E (1-\alpha)  (1-A^2 ) \big)^2 - (1-\alpha)^2 (1-A^2 )^2 \\
&\quad\quad - E (D-1) \Big( 2(1-B^2 ) \left( (1-\alpha)  (1-A^2 ) + \alpha(1-AB)  \right)  \big. \\
&\quad\quad \big. -  (1-\alpha) \left( (A-B)^2 + (1-AB)^2 \right) + \lambda (A-B)(1+B)^2 \Big) \\
&\quad\quad -  (1-E^2 ) \Big( 2 \alpha  \left((1-\alpha)  (1-A^2 ) (1-AB)  \right) + \alpha^2 \left( (A-B)^2 + (1-AB)^2 \right) \\
&\quad\quad  +\lambda (A-B)(1+A) \big((1- \alpha)(1+ A)+\alpha(1+B) \big) \Big)\Big) >  0,\\
&N= \Big( \big( (D-1) (1-B)^2 - (1-\alpha) E (1-A)^2 \big)^2 - (1-\alpha)^2 (1-A)^4 \Big) \\
&\quad\quad - (1-E^2)\Big( \big(\alpha (1-A)(1-B) - \lambda(A-B)\big)^2-2(1-\alpha)(1-A)^2\Big.\\
 &\quad \quad \times\Big.\big(\lambda (A-B) - \alpha (1-A)(1-B) \big) \Big)+ 2  E (D-1) (1-B)^2
\end{align*}
Also let $G$ and $H$ as in \eqref{eq0}.
In addition, for all $n>1$, let
\begin{align*}
 &  N G^2-2 M GH +K H^2 \leq  0
\end{align*}
If $p\in \mathcal{H}_{\mu,n}$ and
\begin{equation}
  \alpha p(z) + (1-\alpha) p^2(z) +  \lambda z p'(z) \prec \frac{1+Dz}{1+Ez}, \label{eq9}
\end{equation}
then
\begin{align*}
 p(z) \prec \frac{1+Az}{1+Bz}.
\end{align*}
\end{lem}
\begin{proof}
In view of  \eqref{eq2}  and \eqref{eq9}, it follows that
\begin{align*}
\Ree \left\{
\psi(q(z),zq'(z))\right\}=\Ree \left\{
\frac
{P(q(z),zq'(z))}{Q(q(z),zq'(z))}\right\}
> 0,
\end{align*}
where
\begin{align*}
 P(q(z),zq'(z))=&(D-1) \big((1+B) q(z)+(1-B)\big)^2\\
& + (1-E) \Big( \alpha \big( (1+A) q(z)+ (1-A)\big)\big((1+B)q(z)+(1-B)\big)\Big)\\
&+ (1-E) \Big((1-\alpha) \big((1+A)q(z)+(1-A)\big)^2 + 2\lambda z (A-B)q'(z) \big)\Big)
\end{align*}
and
\begin{align*}
Q(q(z),zq'(z))=&(D-1) \big((1+B) q(z)+(1-B)\big)^2 \\
&- (1+E) \Big( \alpha \big( (1+A) q(z)+ (1-A)\big)\big((1+B)q(z)+(1-B)\big)\Big)\\
&- (1+E) \Big((1-\alpha) \big((1+A)q(z)+(1-A)\big)^2 + 2\lambda z (A-B)q'(z) \Big).
\end{align*}
Define $\psi:\mathbb{C}^2\rightarrow \mathbb{C}$ by
\begin{align*}
\psi(r,s)=:\frac{P(r,s)}{Q(r,s)},
\end{align*}
where
\begin{align*}
 P(i \rho, \sigma)=&(D-1)(1-B)^2 + \alpha (1-E) (1-A)(1-B) + (1-E)(1-\alpha)(1-A)^2  \\
&+ 2 \lambda (1-E) (A-B)\sigma -\big((D-1)(1+B)^2 + \alpha (1-E)(1+A)(1+B) \\
& + (1-\alpha)(1-E)(1+A)^2\big) \rho^2 +2 \big( (D-1)(1-B^2) + \alpha (1-E)(1-A B)\\
& +  (1-E)(1-\alpha)(1-A^2)\big) i \rho
\end{align*}
and
\begin{align*}
 Q(i \rho, \sigma)=&(D-1)(1-B)^2 - \alpha (1+E) (1-A)(1-B) - (1+E)(1-\alpha)(1-A)^2 \\
& - 2 \lambda  (1+E) (A-B)\sigma -\big((D-1)(1+B)^2 - \alpha (1+E)(1+A)(1+B)  \\
&  - (1-\alpha)(1+E)(1+A)^2 \big) \rho^2+2 \big( (D-1)(1-B^2) - \alpha (1+E)(1-A B)\\
&-  (1+E)(1-\alpha)(1-A^2)\big) i \rho.
\end{align*}
Then $\psi $ is continuous of r and s on $D =:\mathbb{C}^2-\{(r,s):Q(r,s)=0 \}$. Note that
 $(1, 0) \in D $ and $\Ree\left\{\psi (1, 0)\right\}> 0$. It also follows that for all
 $(i \rho , \sigma) \in D,$ $ \left\{\Ree \psi(i \rho, \sigma)\right\}=\Ree \left\{
P(i \rho, \sigma)/Q(i \rho, \sigma)\right\}.$

Suppose
\begin{align*}
& a = (D-1)(1-B)^2 + \alpha (1-E) (1-A)(1-B) + (1-E)(1-\alpha)(1-A)^2, \\
& b = 2 \lambda (1-E) (A-B), \\
& c = - (D-1)(1+B)^2 - \alpha (1-E)(1+A)(1+B) - (1-\alpha)(1-E)(1+A)^2, \\
& d = 2 (D-1)(1-B^2) +2 \alpha (1-E)(1-A B) +  2 (1-E)(1-\alpha)(1-A^2),\\
& e=  (D-1)(1-B)^2 - \alpha (1+E) (1-A)(1-B) - (1+E)(1-\alpha)(1-A)^2 ,\\
& f = - 2 \lambda  (1+E) (A-B), \\
& g = - (D-1)(1+B)^2 + \alpha (1+E)(1+A)(1+B) + (1-\alpha)(1+E)(1+A)^2 \\
& h = 2 (D-1)(1-B^2) -2 \alpha (1+E)(1-A B) -  2 (1+E)(1-\alpha)(1-A^2).
\end{align*}
Then $ \Ree \psi(i \rho, \sigma)= \Ree \left\{(a + b\sigma + c \rho^2 + d i \rho)/(e + f\sigma + g \rho^2 + h i \rho )\right\}$.
For $\Ree \psi(i \rho, \sigma)<0$ we need to claim
\[ae + (af+be)\sigma + \big((ag+ce+hd) + (bg+cf)\sigma\big)\rho^2 + cg\rho^4 + fb \sigma^2 < 0.\]
Since $\sigma \leq  -1/2$, then
\begin{align}\label{eq2.9}
   ag+ce+hd+(bg+cf)\sigma & \geq ag+ce+hd-\frac{1}{2}(bg+cf)\nonumber \\
  &=  2 \Big[ \big( (D-1) (1-B^2 ) - E (1-\alpha)  (1-A^2 ) \big)^2 - (1-\alpha)^2 (1-A^2 )^2 \nonumber \\
&  - E (D-1) \Big( 2(1-B^2 ) \left( (1-\alpha)  (1-A^2 ) + \alpha(1-AB)  \right)  \big. \nonumber \\
&  \big. -  (1-\alpha) \left( (A-B)^2 + (1-AB)^2 \right) + \lambda (A-B)(1+B)^2 \Big) \nonumber \\
&  -  (1-E^2 ) \Big( 2 \alpha  \left((1-\alpha)  (1-A^2 ) (1-AB)  \right) + \alpha^2 \left( (A-B)^2 + (1-AB)^2 \right) \nonumber \\
&   +\lambda (A-B)(1+A) \big((1- \alpha)(1+ A)+\alpha(1+B) \big) \Big)\Big]>0,
\end{align}
if $bg+cf<0$ holds.
We have $\rho^2   \leq  - \big( \left( 2 (2+\mu') \sigma \right) / \left(n(2+\mu')+ (2 - \mu') \right) + 1  \big)$ where $\mu'=2 \mu/(A-B)$. Therefore if \eqref{eq2.9} holds and
$ \big((D-1) (1+B)^2  - (1-\alpha) E \big.$ $\big.(1+A)^2 \big)^2
   - \alpha (1+A) (1+B) \big( 2 E (1+B)^2 + \alpha (1-E^2) (1+A) (1+B) + 2 (1-\alpha)(1-E^2)\big.$ $\big. (1+A)^2 \big)
 - (1-\alpha)^2 (1+A)^4>0$. Similar calculation as previous lemma shown that
\[N G^2-M GH +K H^2 \leq  0\]
which leads to the desired result.
\end{proof}

\begin{theorem}\label{thm2.4}
Let the conditions of Lemma \ref{lem2.5} hold. If $f\in \mathcal{A}_{n,b}$ satisfies
\[\frac{zf'(z)}{f(z)}+\beta \frac{z^2f''(z)}{f(z)} \prec \frac{1+Dz}{1+Ez},\]
then $f\in S^*[A,B]$.
\end{theorem}
\begin{proof}
For $\lambda = \beta$, $\alpha=1-\beta$, and $p(z)=zf'(z)/f(z),$ where $f\in \mathcal{A}_{n,b}$ and $0<\mu=nb\leq A-B$ in Lemma \ref{lem2.5}, we get the required result.
\end{proof}
\begin{theorem}\label{thm2.5}
Let the conditions of Lemma \ref{lem2.5} hold. If $f\in \mathcal{A}_{n,b}$ satisfies
\[f'(z)+\beta zf''(z)\prec \frac{1+Dz}{1+Ez},\]
then $f'(z) \prec \frac{1+Az}{1+Bz}.$
\end{theorem}
\begin{proof}
For $\lambda = \beta$, $\alpha=1$, and $p(z)=f'(z),$ where $f\in \mathcal{A}_{n,b}$ and $0<\mu=(n+1)b\leq A-B$ in Lemma \ref{lem2.5}, we get the required result.
\end{proof}

\begin{lem}\label{lem2.6}
Suppose that $-1 \leq  B < A \leq  1$, $-1 \leq  E < D \leq  1$,
\begin{align*}
&K=(1+A)^2(1+B)^2 \big( \left( (D-1)  - \alpha E \right)^2 - \alpha^2 \big)-(1- \alpha)(1+A)^3  (1-E^2) \\
& \quad\quad\times \big( 2 \alpha (1+B) + (1-\alpha) (1+A) \big)\geq0, \\
&L= \alpha \Big( E (D-1)(1+B) + (1-E^2) \big(\alpha (1+B) +  (1-\alpha)(1+A) \big) \Big) < 0, \\
& M=   \Big(\big(( (D-1)-\alpha E )^2-\alpha^2 \big) \big( (A-B )^2+(1-AB)^2\big)-(1+A)\big( E (D-1)+\alpha (1-E^2 )\big)\\
&\quad\quad\times \big(2(1 - \alpha)(1-A)(1-AB)+\alpha(1+B)(A-B) \big)- (1 - \alpha)(1-E^2)(1+A) \\
& \quad\quad\times \big( (1 - \alpha)(1-A)(1-A^2) + \alpha (1+A)(A-B) \big)\Big) > 0, \\
& N= (1-A)^2 (1-B)^2 \Big( \big( (D-1)  -  \alpha E  \big)^2 - \alpha^2 \Big)+ (E^2 -1 )\big( (1-\alpha) (1-A)^2 - \alpha (A-B) \big)^2\\
&\quad\quad + 2 (1-A)(1-B) \big(E(D-1) + \alpha (1-E^2 )\big) \big( \alpha (A-B) - (1-\alpha) (1-A)^2 \big).
\end{align*}
Also let $G$ and $H$ as in \eqref{eq0}. In addition, for all $n>1$, let\begin{align*}
 &  N G^2-2 M GH +K H^2 \leq  0
\end{align*}
If $p\in \mathcal{H}_{\mu,n}$ and
\begin{equation}
(1-\alpha) p(z) + \alpha \left( 1+ \frac{ z p'(z)}{p(z)} \right) \prec \frac{1+Dz}{1+Ez}, \label{eq11}
\end{equation}
then
 \[p(z) \prec \frac{1+Az}{1+Bz}.\]
\end{lem}

\begin{proof}
In view of inequalities \eqref{eq2} and \eqref{eq11}, it follows
\begin{align*}
\Ree \left\{
\psi(q(z),zq'(z))\right\}=\Ree \left\{
\frac
{P(q(z),zq'(z))}{Q(q(z),zq'(z))}\right\}
> 0,
\end{align*}
where
\begin{align*}
 P(q(z),zq'(z))=&(D-1) \big((1+B) q(z) +(1-B)\big)\big((1+A) q(z)  + (1-A)\big)+ \alpha (1-E) \\
&\times \Big( \big( (1+A)  q(z) + (1-A) \big)  \big( (1+B) q(z) +(1-B) \big)+ 2(A-B) zq'(z) \Big)\\
&+ (1-\alpha)(1-E)\big((1+A) q(z) + (1-A)\big)^2
\end{align*}
and
\begin{align*}
Q(q(z),zq'(z))=&(D-1) \big((1+B) q(z) +(1-B)\big)\big((1+A) q(z)  + (1-A)\big)- \alpha (1+E) \\
&\times \Big( \big( (1+A)  q(z) + (1-A) \big) \big( (1+B) q(z) +(1-B) \big) + 2(A-B) zq'(z)\Big) \\
&- (1-\alpha)(1+E)\big((1+A) q(z)  + (1-A)\big)^2
\end{align*}
Define $\psi:\mathbb{C}^2\rightarrow \mathbb{C}$ by
\begin{align*}
\psi(r,s)=:\frac{P(r,s)}{Q(r,s)},
\end{align*}
where
\begin{align*}
 P(i \rho, \sigma)&=(1-A)(1-B) \big((D-1) + \alpha (1-E)\big) + (1-\alpha)(1-E)(1-A)^2+ 2 \alpha (1-E) \\
& \times(A-B)\sigma -\Big((1+A)(1+B) \big((D-1) + \alpha (1-E)\big) - (1-\alpha)(1-E)(1+A)^2\Big) \rho^2  \\
& +\Big( 2(1-AB) \big((D-1) + \alpha (1-E)\big) + 2 (1-\alpha)(1-E) \left(1-A^2\right)\Big) i \rho
\end{align*}
and
\begin{align*}
Q(i \rho, \sigma)&=(1-A)(1-B) \big((D-1) - \alpha (1+E)\big) - (1-\alpha)(1+E)(1-A)^2- 2 \alpha  (1+E)\\
& \times(A-B)\sigma -\Big((1+A)(1+B) \big((D-1) - \alpha (1+E)\big) + (1-\alpha)(1+E)(1+A)^2 \Big) \rho^2 \\
& +\Big( 2(1-AB)  \big((D-1) - \alpha (1+E)\big)  - 2 (1-\alpha)(1+E) \left(1-A^2\right)\Big) i \rho.
\end{align*}
Then $\psi $ is continuous of r and s on $D =:\mathbb{C}^2-\{(r,s):Q(r,s)=0 \}$. Note that
 $(1, 0) \in D $ and $\Ree\left\{\psi (1, 0)\right\}> 0$. It also follows that for all
 $(i \rho , \sigma) \in D,$ $ \Ree \left\{\psi(i \rho, \sigma)\right\}=\Ree \left\{
P(i \rho, \sigma)/Q(i \rho, \sigma)\right\}.$

For ease in computations, let
\begin{align*}
& a =  (1-A)(1-B) \big((D-1) + \alpha (1-E)\big) + (1-\alpha)(1-E)(1-A)^2, \\
& b = 2 \alpha (1-E) (A-B) ,\\
& c = - (1+A)(1+B) \big((D-1) + \alpha (1-E)\big) - (1-\alpha)(1-E)(1+A)^2 ,\\
& d = 2(1-AB) \big((D-1) + \alpha (1-E)\big) + 2 (1-\alpha)(1-E) \left(1-A^2\right), \\
& e=  (1-A)(1-B) \big((D-1) - \alpha (1+E)\big) - (1-\alpha)(1+E)(1-A)^2, \\
& f = - 2 \alpha  (1+E) (A-B), \\
& g = - (1+A)(1+B) \big((D-1) - \alpha (1+E)\big) + (1-\alpha)(1+E)(1+A)^2 ,\\
& h = 2(1-AB) \big((D-1) - \alpha (1+E)\big) - 2 (1-\alpha)(1+E) \left(1-A^2\right).
\end{align*}
Then we need to show $ \Ree \psi(i \rho, \sigma)=ae + (af+be)\sigma + \big((ag+ce+hd) + (bg+cf)\sigma \big)\rho^2 + cg\rho^4 + fb \sigma^2 < 0$.
Since $\sigma \le -1/2$, then
\begin{align}\label{eq2.10}
   ag+ce+hd+(bg+cf)\sigma &\geq ag+ce+hd-\frac{1}{2}(bg+cf)\nonumber \\
  & =2 \Big(\left (( (D-1)  - \alpha E )^2 - \alpha^2 \right) \left( (A-B )^2+(1-AB)^2  \right)- (1+A) \nonumber \\
    & \quad  \times\big( E (D-1)  + \alpha (1-E^2) \big) \big(2(1 - \alpha)(1-A)(1-AB)\big.\nonumber \\
    & \quad \big.+\alpha(1+B)(A-B) \big)- (1 - \alpha)(1-E^2)(1+A) \big( (1 - \alpha)(1-A)(1-A^2)\big. \nonumber \\
    &\quad +\big. \alpha (1+A)(A-B) \big)\Big)>0,
\end{align}
with $ \alpha \Big(  E (D-1)(1+B) + (1-E^2) \big(\alpha (1+B)+  (1-\alpha)(1+A) \big) \Big) < 0$. The proof is obtained by using similar computation as Lemma \ref{lem2.1}.
 \end{proof}

\begin{theorem}\label{thm2.7}
 Let the conditions of Lemma \ref{lem2.6} hold. If $p \in \mathcal{H}_{\mu,n}$ satisfies
\begin{equation}
  1+ \frac{ z p'(z)}{p(z)} \prec \frac{1+Dz}{1+Ez}, \label{eq11}
\end{equation}
then
 \[p(z) \prec \frac{1+Az}{1+Bz}.\]
\end{theorem}
\begin{proof}
The proof of this theorem follows from Lemma \ref{lem2.6} by letting  $\alpha=1.$
\end{proof}

\begin{remark}\label{rem2.7}
When $n=1,$ $\mu=A-B,$ $A=-\alpha,$ $B=-1=E,$ $D=-2\alpha/(1+\alpha),$ $0\leq  \alpha <1$ and $p(z)=f'(z)$ Theorem \ref{thm2.7} reduces to \cite[Theorem 1]{Owa}.
\end{remark}

\begin{theorem}\label{thm2.6}
 Let the conditions of Lemma \ref{lem2.6} hold. If $f\in \mathcal{A}_{n,b}$ satisfies
\[(1-2\alpha)\frac{zf'(z)}{f(z)}+\alpha\left(2+\frac{zf''(z)}{f'(z)}\right) \prec \frac{1+Dz}{1+Ez},\]
then $f\in S^*[A,B]$.
\end{theorem}
\begin{proof}
The proof of this theorem is derived from Lemma \ref{lem2.6} by letting  $p(z)=zf'(z)/f(z),$ where $f\in \mathcal{A}_{n,b}$ and $0<\mu=nb\leq A-B.$
\end{proof}

\end{document}